\documentclass[12pt]{amsart}
\usepackage{graphicx}
\usepackage{amstext}
\usepackage{amsfonts, amssymb}
\usepackage{amsmath}
\usepackage{color}
\usepackage{latexsym}
\usepackage{indentfirst}
\usepackage{amssymb}
\usepackage[utf8]{inputenc}

\advance\textheight by \topskip
\newtheorem{theorem}{Theorem}
\newtheorem{proposition}[theorem]{Proposition}
\newtheorem{lemma}[theorem]{Lemma}
\newtheorem{corollary}[theorem]{Corollary}

\theoremstyle{definition}

\title{On the first eigenvalue of the laplacian on compact surfaces of genus three}
\author{Antonio Ros}     
\thanks{Partially supported by MINECO/FEDER grants no. MTM2017-89677-P and Regional J. Andaluc\'ia grants 
no. P06-FQM-01642 and P18-FR-4049.}

\date{}

\begin{document}

\begin{abstract}
{\small 
For any compact Riemannian surface of genus three $(\Sigma,ds^2)$ Yang and Yau proved that the product of the first eigenvalue of the Laplacian $\lambda_1(ds^2)$ and the area $Area(ds^2)$ is bounded above by $24\pi$. In this paper we improve the result and we show that 
$\lambda_1(ds^2)Area(ds^2)\leq16(4-\sqrt{7})\pi \approx 21.668\,\pi$. 
About the sharpness of the bound, for the hyperbolic Klein quartic surface numerical computations give
the value  $\approx 21.414\,\pi$. }

\vspace{.1cm}

\end{abstract}

\maketitle

\section{Introduction}

Let $\Sigma$ be a compact orientable surface of genus $g$. We define 
\begin{equation}
\label{lambda1}
\Lambda_1(g)=sup \, \left\{ \lambda_1(ds^2)Area(ds^2)\,\big{/} \, ds^2 \,\,{\rm is \, \, a \, \, Riemannian \, \, metric\, \, on} \, \, \Sigma\right\}
\end{equation}
where $\lambda_1(ds^2)$ and $Area(ds^2)$ are the first positive eigenvalue of the Laplacian and the area of $ds^2$, respectively. 
\vspace{.2cm}
Yang and Yau \cite{yy} gave the upper bound  

\[
\Lambda_1(g) \leq \left[\frac{g+3}{2}\right]8\pi,
\]
where $[x]$ denotes the integer part of $x$.
The argument uses as test functions branched conformal maps between $(\Sigma,ds^2)$ and the round sphere $S^2(1)\subset\mathbb{R}^3$. 

\vspace{.2cm}

Among the known results about this eigenvalue functional, we remark the following ones:

\vspace{.2cm}

$\cdot$ Hersch, \cite{hersch}, $\Lambda_1(0) = 8\pi$ and the equality holds only for constant curvature metrics. 
Hersch's theorem relates the first eigenvalue of the Laplacian and the conformal geometry of the sphere and this idea is at the origin of \cite{yy}, Li and Yau \cite{liyau} and the rest of the results that appear in this paper.

\vspace{.2cm}

$\cdot$ Nadirashvili, \cite{nadir},  $\Lambda_1(1) =  \frac{8}{\sqrt{3}}\pi^2 \approx  14.510\pi$. The equality holds for the flat equilateral torus.

\vspace{.2cm}

$\cdot$ Nayatani and Shoda, \cite{nayatani}, $\Lambda_1(2) =  16\pi$. The equality holds for the Bolza spherical surface and some other branched spherical metrics. 

\vspace{.4cm}

Concerning the metrics that attain the supremum $\Lambda_1(g)$, from Petrides \cite{Petri}
and  Matthiesen and Siffert \cite{matthie},  
for any 
genus $g$ we have 
an {\it extremal} metric $ds^2$ with conical singularities on $\Sigma$: 
$\lambda_1(ds^2)Area(ds^2)=\Lambda_1(g)$ and there is a branched minimal immersion 
$f:\Sigma\longrightarrow S^m(1)$, $m\geq 2$, given by first eigenfunctions of $ds^2$.

\vspace{.4cm}

In this paper we will consider surfaces of genus three. From \cite{yy} we have $\Lambda_1(3) \leq  24\pi$ and
 here we will prove the following result

\begin{theorem}
\label{main} For compact orientable 
Riemannian surfaces of genus $3$ one has the eigenvalue inequality
\[
\Lambda_1(3) \leq  16(4-\sqrt{7})\hspace{.02cm}\pi \approx 21.668\hspace{.02cm}\pi.
\] 
\end{theorem}

\vspace{.2cm}

We outline the idea of the proof. Bourguignon, Li and Yau, \cite{bourgui}, 
give an eigenvalue bound for algebraic Kaehler manifolds which extends \cite{yy}
 (we will consider just the case of holomorphic curves in the complex projective plane).
There is a natural isometric embedding of the complex projective plane 
$\mathbb{C}{\mathrm P}^2$ with the Fubini-Study metric
 in the Euclidean space of $3\times 3$ Hermitian matrices ${H\hspace{-.03cm}M}(3)$.
Given a compact Riemannian surface $(\Sigma,ds^2)$ we consider a holomorphic embedding of the induced 
Riemann surface $\Sigma$ in the complex projective plane (i.e. we see $\Sigma$ as a complex curve). 
Then they showed that there is a projective transformation $P$ 
on $\mathbb{C}{\mathrm P}^2$  so that the mean value of the map
$h:(\Sigma,ds^2)\longrightarrow \mathbb{C}{\mathrm P}^2\overset{P}{\longrightarrow}\mathbb{C}{\mathrm P}^2
 \longrightarrow {H\hspace{-.03cm}M}\hspace{-.03cm}(3)$
is  proportional to the identity matrix. In this way they prove that the first eigenvalue of the Laplacian 
of $(\Sigma,ds^2)$ is bounded in terms of the energy of $h$ and their result follows because 
this energy depends on the algebraic geometry of the curve $\Sigma$. 

\vspace{.2cm}

 However, when applied to a general compact surface of genus $3$, as $\Sigma$ corresponds to a quartic curve 
in $\mathbb{C}{\mathrm P}^2$, we obtain $\Lambda_1(3)\leq 24\pi$, the same as in \cite{yy}. 
In our argument, to get the eigenvalue estimate we use as a test function not only the inclusion map of the embedding of $\Sigma
\subset\mathbb{C}{\mathrm P}^2$ 
in ${H\hspace{-.03cm}M}(3)$ but we combine it with its Gauss 
map, \cite{ros1}. In this way we get a conformal spherical map from $(\Sigma,ds^2)$ to 
$H\hspace{-.03cm}M(3)$ that, after suitable projective deformation, has controlled mean value. 
From the energy of this new test map, we obtain a tighter upper bound of the functional 
$\lambda_1(ds^2)Area(ds^2)$.

\vspace{.2cm}

The author would like to thank F. Urbano for his valuable comments about the paper.

\vspace{.2cm}

\section{Preliminaries} 

\subsection{The complex projective plane.} 

\hspace{.2cm}

\vspace{.2cm}

Let $H\hspace{-.03cm}M\hspace{-.03cm}(3)=\{A\in gl(3,\mathbb{C}) \, / \, \bar{A}=A^t \}$ the space of 
$3\hspace{-.08cm}\times \hspace{-.08cm}3$-Hermitian matrices with the Euclidean metric 
\[
\langle A,B\rangle = 2\hspace{.05cm} tr\hspace{.01cm} AB \hspace{.5cm} \forall A,B\in H\hspace{-.03cm}M\hspace{-.03cm}(3),
\]
$I\in H\hspace{-.03cm}M\hspace{-.03cm}(3)$ the identity matrix and ${H\hspace{-.03cm}M}_1\hspace{-.03cm}(3)=\{A\in H\hspace{-.03cm}M\hspace{-.03cm}(3)\, /\,  tr\hspace{.01cm} A=1\}\simeq\mathbb{R}^8$ be the  hyperplane given by the trace $1$ restriction.

\vspace{.2cm}

The submanifold $\mathbb{C}{\mathrm P}^2=\{A\in {H\hspace{-.03cm}M}_1\hspace{-.03cm}(3)\, / \, AA=A \}$ with the induced metric is isometric to the complex projective plane with the Fubini-Study metric of constant holomorphic sectional curvature $1$, see Ros \cite{ros1}. This embedding appears naturally from several points of view: Immersions by the first eigenfunctions, equivariant embeddings of homogeneous spaces, immersions with planar geodesics, etc. 
The approach used here follows Tai \cite{tai}.  

\vspace{.2cm}

The action of the unitary group $U(3)$ on $\mathbb{C}{\mathrm P}^2$ is given by $(P,A)\mapsto \bar{P}^{t}\hspace{-.1cm}AP$, 
where $P\in U(3)$ and $A\in \mathbb{C}{\mathrm P}^2$. Hence the embedding of $\mathbb{C}{\mathrm P}^2$ in 
${H\hspace{-.03cm}M}_1\hspace{-.03cm}(3)$ is $U(3)$-equivariant. 

\vspace{.2cm}

We denote by $\tilde{\sigma}$ the second fundamental form of $\mathbb{C}{\mathrm P}^2 \subset 
{H\hspace{-.03cm}M}_1\hspace{-.03cm}(3)$. It is a symmetric tensor such that, for every point 
$A\in \mathbb{C}{\mathrm P}^2$, maps the tangent vectors $V,W\in T_A \mathbb{C}{\mathrm P}^2$ into the normal component of the immersion
\[
\tilde{\sigma}(V,W) \in  T_A^\perp \mathbb{C}{\mathrm P}^2.
\]

The mean curvature vector of $\mathbb{C}{\mathrm P}^2$ in ${H\hspace{-.03cm}M}_1\hspace{-.03cm}(3)$ is given by  
\[
\widetilde{H}=\frac{1}{4} tr\, \tilde{\sigma}=-\frac{3}{4}[A-\frac{1}{3}I]
\]
and $\mathbb{C}{\mathrm P}^2$ is a minimal submanifold in the sphere $S^7$ of 
${H\hspace{-.03cm}M}_1\hspace{-.03cm}(3)$ of center $\frac{1}{3}I$ and radius $2/\sqrt{3}$. 

\vspace{.1cm}

Among other properties of the embedding we remark the following ones, \cite{ros1}:

\begin{itemize} 

  \item[$\cdot$] Complex projective lines $\mathbb{C}{\mathrm P}^1\subset \mathbb{C}{\mathrm P}^2$ are totally geodesic and,  when viewed in ${H\hspace{-.03cm}M}_1\hspace{-.03cm}(3)$, are given by round 2-spheres of radius one. 
\vspace{.1cm}
  \item[$\cdot$] In the same way, geodesics of $\mathbb{C}{\mathrm P}^2$ are unit circles in 
  ${H\hspace{-.03cm}M}_1\hspace{-.03cm}(3)$. The planar geodesics property characterizes the nicest 
embedding of compact symmetric spaces of rank $1$ in the euclidean space, 
see Little \cite{little} and Sakamoto \cite{sakamoto}.
  \vspace{.1cm} 
 \item[$\cdot$] If $J$ is the complex structure in $\mathbb{C}{\mathrm P}^2$, then 
  $\tilde{\sigma}(JV,JW)=\tilde{\sigma}(V,W)$,  $\forall \,V,W\in T_A \mathbb{C}{\mathrm P}^2$.
 \vspace{.1cm} 
 \item[$\cdot$]  $\widetilde{\nabla} \tilde{\sigma}=0$, i. e. the second fundamental form is parallel.
\end{itemize}
  
\vspace{.2cm}

Let $A\in HM(3)$. We will use the notation
$A > 0$, resp. $A\geq 0$, when the Hermitian matrix is positive definite, resp. positive semidefinite. 

\vspace{.2cm}

Let ${\mathcal H}$ be the convex hull of $\mathbb{C}{\mathrm P}^2$. Then ${\mathcal H}$ verifies the 
following properties, see \cite{bourgui}:

\vspace{.2cm}
\begin{itemize}

\item[$\cdot$] ${\mathcal H} = \{A\in {H\hspace{-.03cm}M}_1\hspace{-.03cm}(3)\, / \, A\geq 0 \}$, 

\vspace{.2cm}

\item[$\cdot$]   
$int\,{\mathcal H}= \{ A\in {\mathcal H} \, / \, rank\hspace{.02cm} A =3 \} \, = \, 
\{ A\in {H\hspace{-.03cm}M}_1\hspace{-.03cm}(3) \, / \, A> 0\}$, where by $int\,{\mathcal H}$ 
we denote the topological interior of ${\mathcal H}$ in ${H\hspace{-.03cm}M}_1\hspace{-.03cm}(3)$, 

\vspace{.2cm}

\item[$\cdot$] $\partial{\mathcal H} = \{A\in {\mathcal H} \, / \, rank \hspace{.02cm} A \leq 2 \}$,

\vspace{.2cm}

\item[$\cdot$] $\mathbb{C}{\mathrm P}^2\subset \partial{\mathcal H}$ is the set of extremal points of 
$\mathcal H$,

\vspace{.2cm}

\item[$\cdot$] $\partial{\mathcal H}$ is the union of  all the unit $3$-balls in ${H\hspace{-.03cm}M}_1\hspace{-.03cm}(3)$ enclosed by complex projective lines $\mathbb{C}{\mathrm P}^1\subset \mathbb{C}{\mathrm P}^2$.
\end{itemize}

\vspace{.2cm}

The projection of $\mathbb{C}^3-\{0\}$ over $\mathbb{C}{\mathrm P}^2$ 
\[
z\longmapsto \frac{1}{z\bar{z}^t}\bar{z}^t z,
\]
with $z=(z_0,z_1,z_2)$, defines the identification between $\mathbb{C}^3-\{0\}/\sim$, the usual projective 
plane with homogeneous coordinates $[z]=[z_0,z_1,z_2]$, and the submanifold of 
${H\hspace{-.03cm}M}_1\hspace{-.03cm}(3)$.

\vspace{.2cm}

Given a regular matrix $P\in GL(3,\mathbb{C})$, the projective transformation 
$[z]\longmapsto [zP]$, when described in terms of elements of $\mathbb{C}{\mathrm P}^2$, is the map
\begin{equation}
\label{projectivity}
f_P:\frac{1}{z\bar{z}^t}\bar{z}^t z \longmapsto \frac{1}{zP  \bar{P}^t\bar{z}^t }
\bar{P}^t\bar{z}^t zP.
\end{equation}

Following Bourguignon, Li and Yau \cite{bourgui}, we consider projective transformations $f_P$, where 
$P$ is {a} positive definite { matrix} in $HM(3)$: Any other projectivity can be decomposed as  $f_P\circ f_Q$, 
where $P>0$ and $Q$ is a unitary matrix.
Moreover, after multiplying by a positive scalar factor,
we assume that $tr\hspace{.03cm} P=1$. Therefore, up to unitary motions, the space of projective transformations is parametrized by the interior of the convex hull of $\mathbb{C}{\mathrm P}^2$,
\[
f_P:\mathbb{C}{\mathrm P}^2\longrightarrow \mathbb{C}{\mathrm P}^2, \hspace{1cm} P\in int\,{\mathcal H}.
\]  

\vspace{.2cm}

Note that this generalizes the standard fact that conformal transformations of the sphere 
$S^2=\mathbb{C}{\mathrm P}^1$ are parametrized, modulo orthogonal motions, by the open unit 
ball in $\mathbb{R}^3$.

\vspace{.2cm}

\subsection{Geometry of complex curves in $\mathbb{C}{\mathrm P}^2$}  

\hspace{.2cm}

\vspace{.2cm}

Let  $A:\Sigma\longrightarrow \mathbb{C}{\mathrm P}^2\subset H\hspace{-.03cm}M_1\hspace{-.03cm}(3)$ be 
a compact complex curve immersed in the complex projective plane (we will assume that immersion is unbranched). 
For every  point $A\in \Sigma$, the (affine) tangent plane is identified with the corresponding complex 
projective line $\mathbb{C}{\mathrm P}_{\hspace{-.15cm}A}^1$ which is a unit 2-sphere in the Euclidean space. 
We define the {\it Gauss map} $B:\Sigma\longrightarrow HM(3)$ of the curve 
as the vector joining $A$ with its antipodal point $A^-$ in this $2$-sphere,  $B=A^--A$, 
see Figure \ref{tangente}.  
Note that, as a complex line in $\mathbb{C}{\mathrm P}^2$ is determined by two of its points, 
it follows that the vector $B$ determines the tangent $2$-sphere at the point $A\in\Sigma$.
\begin{figure}[h]
\begin{center}
\includegraphics[width=7cm]{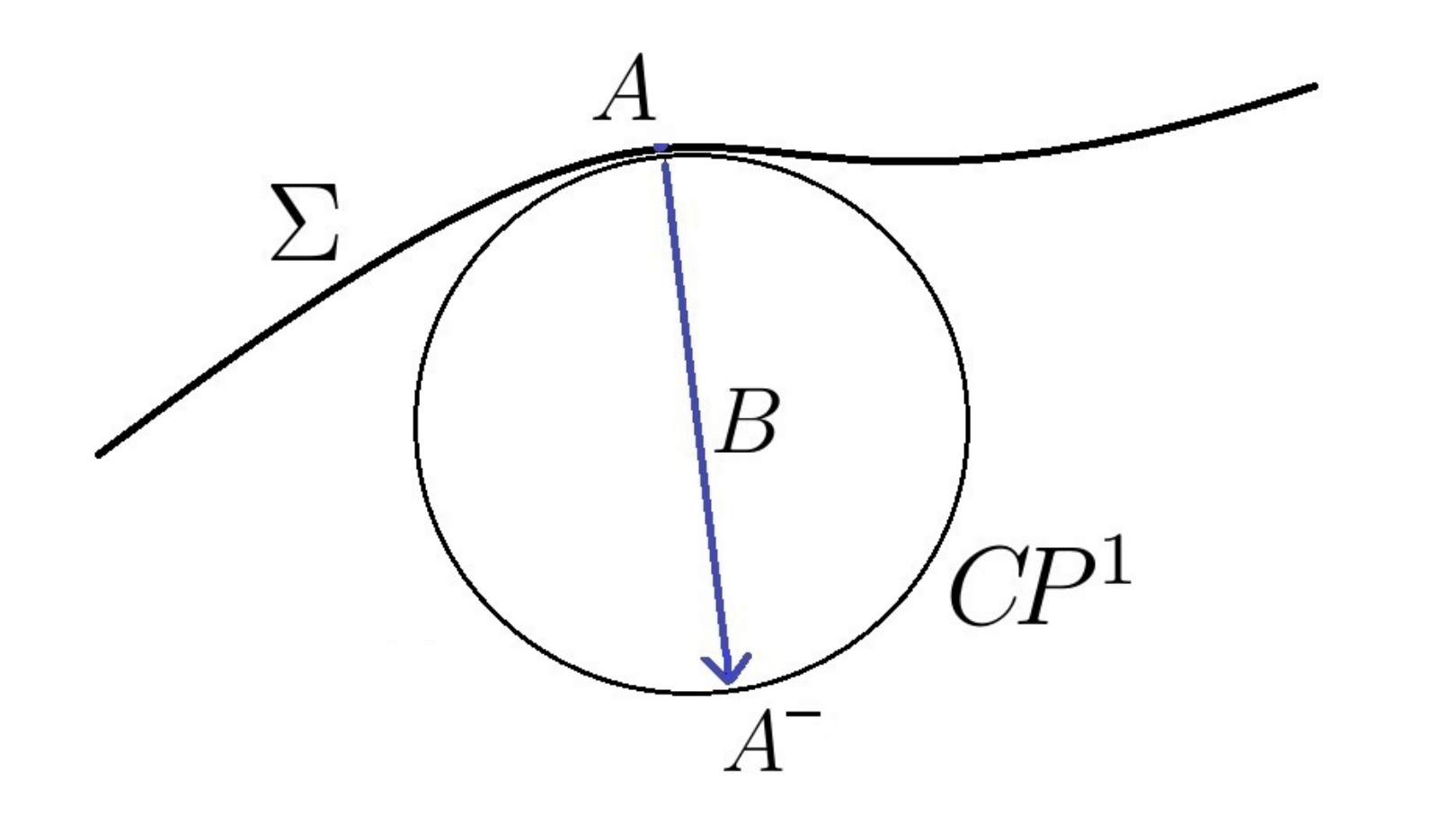}
\caption{}
\label{tangente}
\end{center}
\end{figure}
 On the surface $\Sigma$ we consider, unless otherwise stated, the metric 
$\langle,\rangle$ induced by the immersion and we denote by $K$, $d\Sigma$ and 
$\Delta$ the Gauss curvature, the Riemannian measure and Laplacian of the surface.

\vspace{.2cm}

The Hessian operator of the vector valued map $A:\Sigma\longrightarrow HM(3)$ will be denoted 
by $\nabla^2\hspace{-.1cm}A$. It coincides with the second fundamental form of $\Sigma$ in $HM(3)$ 
and it can decomposed as
\[
\nabla^2 \hspace{-.1cm}A(X,Y)=\sigma(X,Y)+\tilde{\sigma}(X,Y) \hspace{.8cm} \forall\,  X, Y\in T_A\Sigma,
\]
where $\sigma$ is the second fundamental form  of $\Sigma$ in $\mathbb{C}{\mathrm P}^2$, it is normal 
to $T_A\Sigma$ but tangent to $T_A\mathbb{C}{\mathrm P}^2$, and $\tilde{\sigma}$ is normal to 
$\mathbb{C}{\mathrm P}^2$.

 \vspace{.2cm}

The compatibility between the complex structure $J$ and the second fundamental form $\sigma$ 
is given by
\begin{equation}
\label{J}
\sigma(JX,Y)=\sigma(X,JY)=J\sigma(X,Y)
\end{equation}
and, as a consequence, for every orthonormal basis 
$E_1,E_2$ of $T_A\Sigma$,  with $E_2=JE_1$, there is an orthonormal basis  $E_3,E_4$ of $T_A^\perp\Sigma$,
$E_4=JE_3$, and a real parameter $\lambda$ such that  $\sigma(E_1,E_1)=\lambda E_3$. Moreover, 
 $\sigma$ is determined by these data.
In particular, we have
\begin{equation}
\label{cuenta1}
\sum_{i=1,2}\big\langle \sigma(E_i,X),\sigma(E_i,Y) \big\rangle  
= \frac{1}{2}|\sigma|^2\big\langle X,Y\big\rangle
\end{equation}
and the Gauss equation 
\begin{equation}
\label{gauss}
K=1-\frac{1}{2}|\sigma|^2.
\end{equation}

As complex curves are minimal surfaces in the complex projective plane, it follows that the 
Laplacian of the immersion map $A:\Sigma\longrightarrow HM(3)$ is equal to
\begin{equation}
\Delta A =\sum_i \nabla^2\hspace{-.1cm}A(E_i,E_i) 
= \tilde{\sigma}(E_1,E_1)+\tilde{\sigma}(E_2,E_2) = 2\vec{H} = B,
\label{B}
\end{equation}
$\vec{H}$ being the mean curvature vector of the immersion. With respect to the equality $\Delta A=B$, 
as both functions depend only on the tangent plane $T_A\Sigma$, it is enough to check it for the sphere 
$\mathbb{C}{\mathrm P}^1$. 

\vspace{.2cm} 

From Lemma 3.2 in  \cite{ros1} (note that the Laplacian $\Delta$ considered there is the opposite of 
the one used here), we have 
\begin{equation}
\left\{
\begin{array}{ccc}
|I|^2 = 6\hspace{1cm} & \langle A,I\hspace{.02cm}\rangle = 2 \hspace{1cm} 
& |A|^2=2 \vspace{.1cm}
\\
\langle B,I \hspace{.02cm}\rangle =0 \hspace{1cm}& \langle B,A\rangle=-2  \hspace{1cm} &|B|^2= 4 
\vspace{.1cm}
\\
\langle \Delta B,I\hspace{.03cm}\rangle=0 \hspace{1cm}& \langle \Delta B,A\rangle= 4  
\hspace{1cm}& \langle \Delta B,B\rangle= - 8 - 2|\sigma|^2
\end{array}
\right.
\label{calculos}
\end{equation}

\vspace{.2cm}

The $degree$ of the complex curve $A:\Sigma\longrightarrow \mathbb{C}{\mathrm P}^2$ is an 
{\it algebro-geometric/topological} invariant given by a positive integer $d$ satisfying the 
following properties:

\vspace{.2cm}

a)\hspace{.01cm}  If $\omega_\Sigma$ and $\omega_{\mathbb{C}{\mathrm P}^2}$ are the Kaehler 2-forms of 
$\Sigma$ and $\mathbb{C}{\mathrm P}^2$, respectively, then the pullback image 
between the cohomology spaces
$A^*:H^2(\mathbb{C}{\mathrm P}^2, \mathbb{C})\longrightarrow H^2(\Sigma, \mathbb{C})$ maps the class 
$[\omega_{\mathbb{C}{\mathrm P}^2}]$ into $ d\hspace{.04cm}[\omega_\Sigma]$. In particular 
$Area(\Sigma)=4\pi d$.

\vspace{.2cm}

b)\hspace{.01cm}  Any complex line intersects the curve $\Sigma$ at exactly $d$ points (counted 
with multiplicity). As a consequence, in the embedded case the curve is given by a homogeneous polynomial equation of degree $d$. 

\vspace{.2cm}

c) \hspace{.01cm} If  $g$ is the genus of $\Sigma$, the integral of (\ref{gauss}) combined with 
the Gauss-Bonnet theorem give
\begin{equation}
\int_\Sigma 1 \hspace{.05cm} d\Sigma= 4\pi d, \hspace{1cm}
\int_\Sigma |\sigma|^2 d\Sigma= 8\pi (g+d-1). 
\label{genus}
\end{equation}


\vspace{.2cm}

For $a \in\mathbb{R}$ we define the map
\begin{equation}
\label{phia}
\phi_a : \Sigma\longrightarrow {H\hspace{-.03cm}M}_1\hspace{-.03cm}(3), \hspace{1cm} \phi_a  = A + aB.
\end{equation}

Below we collect some of its properties.  

\vspace{.1cm}

\begin{lemma} 
\label{lema1}
(a)\,  {\it The map $\phi_a$ is spherical valued},
\[
|\phi_a - \frac{1}{3}I|^2= \frac{4}{3}( 3a^2-3a+1).
\]

\vspace{.1cm}

(b) \, {\it $\phi_a:\Sigma \longrightarrow {H\hspace{-.03cm}M}_1\hspace{-.03cm}(3)$ is  a conformal 
immersion $\forall a\neq \frac{1}{2}$. If $\Sigma\neq \mathbb{C}{\mathrm P}^1$, the map $\phi_{\frac{1}{2}}$ 
is a branched conformal immersion and the branch points lie at the inflection points of $\Sigma$.}

\vspace{.1cm}

(c) {\it The energy of $\phi_a$ is} 
\[
\int_\Sigma |\nabla\phi_a |^2 d\Sigma = 8\pi \big\{ 2(3d+g-1)a^2 - 4\hspace{.01cm}d \hspace{.01cm}a  + d    \big\}. 
\]
In particular, the energy remains invariant under projective transformations.
\end{lemma}

\begin{proof}
(a) Using  (\ref{calculos}) we get 
\[
|\phi_a-\frac{1}{3}I|^2 = |a B +(A-\frac{1}{3}I)|^2 
= a^2|B|^2 + 2a \langle B, A\rangle + |A-\frac{1}{3}I|^2
= \frac{4}{3}(3a^2-3a+{1}).
\]

\vspace{.1cm}

Now we prove (b). The differential of the mean curvature 
vector $\vec{H}$, \cite{ros1} p. 438, is given by
\[
\nabla B(X)= 2\nabla \vec{H}(X)=2\sum_i \tilde{\sigma}(\sigma(E_i,X),E_i) - 2X
\]
and so
\[
\nabla\phi_a(X)=  (1-2a)X + 2a\sum_i \tilde{\sigma}(\sigma(E_i,X),E_i), \hspace{1cm}X\in T_A\Sigma. 
\]

If we denote by $\tau$ the symmetric tensor induced by $\phi_a$,
\[
\tau(X,Y)=
\big\langle \nabla\phi_a(X),  \nabla\phi_a(Y)\big\rangle,  \hspace{1cm} X,Y\in T_A\Sigma,
\]
we have 
\[ 
\tau(X,Y)=  (1-2a)^2\big\langle X,Y\big\rangle  + 
4a^2\sum_{i,j} \big\langle\tilde{\sigma}(\sigma(E_i,X),E_i),\tilde{\sigma}(\sigma(E_j,Y),E_j)\big\rangle.
\]

From (\ref{J}) we conclude that admissible second fundamental forms $\sigma$ depend, up to 
isometries in $U(3)$, on a real parameter. In particular, for any $A\in\Sigma$, there is  
$\lambda\in\mathbb{R}$ such that
\[
\sigma(X,Y)=\lambda \,\sigma_0(X,Y),
\]
where $\sigma_0$ is the second fundamental form of a suitable rotated image $\mathcal C$ of the round complex conic 
$z_0^2+z_1^2+z_2^2=0$: $A\in {\mathcal C}$,
$T_A \hspace{.02cm}{\mathcal C} = T_A\Sigma$ and $\mathcal C$ is isometric to the sphere 
$S^2\subset\mathbb{R}^3$ of radius $\sqrt{2}$.
\vspace{.2cm}

Therefore, in order to prove that $\phi_a:\Sigma\longrightarrow {H\hspace{-.03cm}M}_1\hspace{-.03cm}(3)$ 
is a conformal map, it is enough to show that the same is true for conic map
 $\phi_a{:\mathcal C}\longrightarrow {H\hspace{-.03cm}M}_1\hspace{-.03cm}(3)$:
As this follows directly from the fact that the group of unitary transformations of  
$\mathbb{C}{\mathrm P}^2$ that fix $A$ and $\mathcal C$ is $1$-dimensional (these transformations 
are the rotations of $S^2$ around one of its axis), we deduce that the metric $\tau$ is 
invariant under these unitary transformations and so it is proportional to euclidean one. 

\vspace{.2cm}

For the last part of the assertion in (b), note that, as $ |\phi_a-\frac{1}{3}I|^2$ is constant 
and $\Delta A=B$, we have
\[
|\nabla \phi_a|^2 = - \big\langle \Delta\phi_a, \phi_a\big\rangle =
- \big\langle B+ a\hspace{.03cm} \Delta B, A+ a\hspace{.03cm} B\big\rangle =
\]
\[
-a^2\big\langle \Delta B,B\big\rangle -2a\,|B|^2 -  \big\langle B,A\big\rangle
\]
and using (\ref{calculos}) we obtain
\begin{equation}
| \nabla\phi_a|^2 = 2(1-2a)^2 + 2a^2|\sigma|^2. 
\label{b}
\end{equation}

The inflection points of $\Sigma$ are the points where $\sigma=0$ and (b) is proved.

\vspace{.1cm}  

Finally, (c) follows from (\ref{genus}) and (\ref{b}). 
\end{proof}

\subsection{Quartic curves in the complex projective plane}
\label{quartic}

\hspace{.2cm}

\vspace{.2cm}

Let $\Sigma$ be a compact non hyperelliptic Riemann surface of genus $3$. It is known, 
see \cite{farkas} p. 136,  that the canonical mapping 
\[
[\omega_0,\omega_1,\omega_2]: \Sigma\longrightarrow \mathbb{C}{\mathrm P}^2, 
\]
$\omega_0,\omega_1,\omega_2$ being a basis of the space of holomorphic $1$-forms, 
is an embedding of degree $4$.

\vspace{.2cm}

So $\Sigma$ can be seen as a smooth algebraic quartic curve embedded in the complex projective plane. 
With the induced metric, we have 
\begin{equation}
Area(\Sigma) = 16\pi, \hspace{1cm}  \int_\Sigma|\sigma|^2 d\Sigma= 48\pi.
\label{genus3}
\end{equation}

An important example is the {\it Klein's quartic}, given by the equation
\begin{equation}
\label{klein1}
{\mathcal K} =\left\{[z_0,z_1,z_2]\in \mathbb{C}{\mathrm P}^2\, /\,  z_0^3z_1+z_1^3z_2+z_2^3z_0=0\right\}.
\end{equation}

The Hurwitz's theorem says that the order of the group of holomorphic transformations of a Riemann 
surface of genus $3$ is less than or equal to $168$, \cite{farkas} p. 258, and  the Klein's quartic 
is characterized as the unique case which attains the equality, \cite{karcher}. 

\vspace{.2cm}

\subsection{The first eigenvalue of compact surfaces}
\label{first}

\hspace{.2cm}

\vspace{.2cm}

A Riemannian metric $ds^2$ on a compact orientable surface $\Sigma$ induces a Riemann surface 
structure. The Riemannian measure and the Laplacian operator acting on smooth functions will be 
denoted by $d\mu$ and $\Delta^{ds^2}$. We will also consider the area 
\[
Area(ds^2)=\int_\Sigma 1\, d\mu
\]
and the eigenvalues of the Laplacian
\[
0=\lambda_0(ds^2) < \lambda_1(ds^2)\leq \lambda_2(ds^2) \leq \cdots.
\]

The first eigenvalue $\lambda_1(ds^2)$, written in terms of the Rayleigh quotient, is given by 

\[
\lambda_1(ds^2) = \inf \left\{ \frac{\int_\Sigma |\nabla u|^2 d\mu}{\int_\Sigma u^2 d\mu}\, \left/\right. \, 
u\in C^1(\Sigma)-\{0\}, 
\,  \int_\Sigma u \,d\mu=0 \right\}.
\]

\vspace{.2cm}

Yang and Yau, \cite{yy}, showed that if the Riemann surface induced by $(\Sigma,ds^2)$ admits 
a nonconstant meromorphic map 
$(\Sigma,ds^2)\longrightarrow \mathbb{C}\cup\{\infty\}$ of degree $d$, then 
$\lambda_1(ds^2)Area(ds^2)\leq 8\pi d$. In particular,

\vspace{.1cm}

\begin{itemize}
  
\item[$\cdot$]  (Hersch, \cite{hersch}) For any metric $ds^2$ on the sphere $S^2$, 
$\lambda_1(ds^2)Area(ds^2) \leq 8\pi$ and the equality holds just for metrics of constant curvature. 

\vspace{.2cm}

\item[$\cdot$]   If $(\Sigma,ds^2)$ is hyperelliptic, then $\lambda_1(ds^2)Area(ds^2)\le 16\pi$.
\vspace{.2cm}
\item[$\cdot$]   If $g$ is the genus of $\Sigma$, then for any metric on $\Sigma$, we have
\[
\lambda_1(ds^2)Area(ds^2)\leq  \left[\frac{g+3}{2}\right]8\pi.
\]
\end{itemize} 

\vspace{.1cm}

The invariant $\Lambda_1(g)$ is defined as the supremum of $\lambda_1(ds^2)Area(ds^2)$, 
for any metric $ds^2$ on a compact surface of genus $g$.
The sharp bound is known for $g=0$, $\Lambda_1(0)=8\pi$,  and  $g=1,2$:

\begin{itemize}
  \item[$\cdot$]  (Nadirashvili, \cite{nadir}) 
  $\Lambda_1(1)= 8\pi^2/\sqrt{3}$. The only extremal surface is the flat equilateral torus. 

\vspace{.2cm}

  \item[$\cdot$]  (Nayatani and Shoda, \cite{nayatani})   $\Lambda_1(2)=16\pi$. The equality holds for Bolza surface and some other hyperelliptic branched spherical metrics.  In this case any extremal metric is a spherical one with branched singularities, see Proposition \ref{nueva} below. 

\vspace{.2cm}

  \item[$\cdot$]  (Karpukhin, \cite{karpu2}) For any $g\geq 3$ one has  
  $\Lambda_1(g)< \left[ {(g+3)}/{2}\right]8\pi$.

  \end{itemize}

\vspace{.1cm}

 {\it Extremal metrics} on $\Sigma$ are those which maximize the first eigenvalue functional. 
 Petrides \cite{Petri} and Matthiesen and Siffert \cite{matthie} show that
for any genus there is a metric $ds^2$ with conical singularities which satisfies 
$\lambda_1(ds^2)Area(ds^2)=\Lambda_1(g)$. Moreover this metric admits an isometric branched minimal immersion in a sphere $f:(\Sigma,ds^2)\longrightarrow S^m(1)$ by the first eigenfunctions of $ds^2$.
For other related results see  Montiel and Ros \cite{montiel}, 
El Soufi and Ilias \cite{elsoufi}, Fraser and Schoen \cite{fraser}, 
Nadirashvili and Sire \cite{nadir1}, Kokarev \cite{kokarev},
Cianci, Karpukhin and Medvedev \cite{karpu1}
 and Matthiesen and Petrides \cite{matthiepetrides}.

\vspace{.2cm}

For genus $g=3$ Yang and Yau, \cite{yy}, gave the bound $\Lambda_1(3)\leq 24\pi$ and we will show in 
Theorem \ref{main} that it can be improved to $\Lambda_1(3) \leq 16(4-\sqrt{7})\pi\approx 21.668\pi$.
 Cook \cite{cook}, using a point of view similar to the one of \cite{nadi2}, give a number of results 
on the eigenvalues of the {\it hyperbolic Klein quartic} surface $({\mathcal K},ds^2_{-1})$ 
(the conformal metric of constant negative curvature). In particular, \cite{cook} p. 131, the numerical 
experimental value
\[
\lambda_1(ds^2_{-1}) Area (ds^2_{-1}) \approx 21.414\hspace{.02cm}\pi.
\]

\vspace{.2cm}

Every Riemann surface of genus $\geq 2$ admits a unique hyperbolic metric and, according Theorem \ref{main},  
for the Klein quartic this metric is close to reaching the supremum.
However, hyperbolic metrics are never extremal for any genus: Bryant \cite{bryant} shows that there 
are no local isometric minimal immersions from the hyperbolic plane into a sphere $S^m(R)$ of radius $R>0$. 
On the contrary, it is conjectured that for $g=3$ the Klein's quartic attains the maximum among hyperbolic
 metrics, \cite{cook} p. 99.


\vspace{.2cm}

If $ds^2$ is an {extremal metric} on a compact surface $\Sigma$, then 
the metric has finitely many conical singularities and it admits and isometric minimal immersion 
into a sphere, Petrides \cite{Petri}. Each conical singularity is a branch point of the immersion and the cone angle is a integer multiple of $2\pi$, Kokarev \cite{kokarev} p. 216. In particular, 
the metric induces a Riemann surface structure on the entire $\Sigma$ and $ds^2$ is a conformal metric with finitely many singularities. 

\vspace{.2cm}

In the case of genus two Nayatani and Shoda show that $\Lambda_1(2)=16\pi$ and give some extremal branched spherical metrics (including the Bolza surface), \cite{nayatani}.  
Now we note that this is the general situation.

\begin{proposition} 
\label{nueva}
Let $\Sigma$ be a compact Riemann surface of genus two and \mbox{$ds^2$} an extremal \mbox{metric} on $\Sigma$.  
Then $ds^2$ is the branched spherical metric induced by a degree two conformal map 
\mbox{$\phi:\Sigma\longrightarrow S^2(1)$}. In particular $ds^2$ has exactly six singular points.
\end{proposition}

\begin{proof} On the Riemannian surface $(\Sigma,ds^2)$ we consider $\nabla$, $\Delta$ and $d\mu$, the gradient operator, the Laplacian and the Riemannian measure, respectively.
The induced Riemann surface must be hyperelliptic because it has genus two and thus
there exists a holomorphic map \mbox{$\phi: \Sigma \longrightarrow S^2 (1) \subset \mathbb{R}^3$} with degree two. 
It is easy to show that $\Delta \phi+|\nabla\phi|^2 \phi = 0$. By using the argument of 
Yang and Yau, \cite{yy}, we can assume, after composition with a suitable conformal transformation of $S^2(1)$, that  
the mean value of $\phi$ is equal to zero
\[
\int_\Sigma \phi \, d\mu = 0.
\] 
For the energy of  $\phi$, we have 

\[
\int_\Sigma |\nabla\phi |^2 d\mu =
 8\pi deg(\phi) = 16\pi.
\]
 Moreover, since the image of $\phi$ is in $S^2(1)$, we have
\[
\int_\Sigma |\phi|^2 d\mu = Area(ds^2 ).
\]
Combining these facts yields
\[
\frac{
\int_\Sigma |\nabla\phi|^2 d\mu
}
{
\int_\Sigma |\phi|^2 d\mu
} =
\frac{16\pi}{Area(ds^2)} = \lambda_1(ds^2),
\]
where the second equality follows since $ds^2$ is extremal, \cite{nayatani}. As a consequence the linear coordinates of $\phi$ are eigenfunctions of $\lambda_1(ds^2 )$ and $ds^2$  is induced by the branched minimal
immersion \mbox{$\phi: \Sigma\longrightarrow S^2 (1)$}. Therefore $ds^2$ is a spherical metric with conical singularities and the cone angle is equal to $4\pi$ at each one of these points.
\end{proof}

\vspace{.2cm}

Cianci, Karpukhin and Medvedev,  \cite{karpu1} Theorem 1.4, show that if an extremal metric 
on a Riemann surface $\Sigma$ is given by an holomorphic map  $\phi: \Sigma\longrightarrow S^2 (1)$, 
then any other extremal metric on the same Riemann surface provides a branched conformal immersion 
by first eigenfunctions to the $2$-sphere and the two branched covering maps differ by a composition with a conformal automorphism (not necessarily orthogonal) of $S^2(1)$.


\section{The eigenvalue inequality for surfaces of genus $3$}

In this section, we prepare some material and we prove Theorem \ref{main}. 

\vspace{.2cm}

Let $(\Sigma,ds^2)$ be a Riemannian surface of genus $3$ and 
$
\phi_a:\Sigma\longrightarrow {H\hspace{-.03cm}M}_1\hspace{-.03cm}(3)
$, 
$
\phi_a = A + a \hspace{.02cm} B,
$ 
be the conformal spherical map (\ref{phia}).
We will see in \S\ref{center} that we can modify the surface, by using projective transformations, 
in order its mean value is equal to $\frac{1}{3}I$. Thus we will obtain an upper bound for first 
eigenvalue functional $\lambda_1(ds^2)Area(ds^2)$ in terms of the energy of $\phi_a$. In the case $a=0$, 
this is the approach used by Bourguignon, Li and Yau, \cite{bourgui}.

\begin{lemma}
The distance between the point
$\frac{1}{3}I$ and the convex hypersurface $\partial {\mathcal H}$ is equal to $\sqrt{3}/3$.
\end{lemma}

\begin{proof}
The group $U(3)$ preserves the submanifold $\mathbb{C}{\mathrm P}^2 \subset HM_1 (3)$, the convex body $\mathcal H$ and its
central point  $\frac{1}{3}I$. From (\ref{calculos}), for any $A \in \mathbb{C}{\mathrm P}^2$
we have $|A -\frac{1}{3}I |= 2/\sqrt{3}$.

Let $Q\in\partial{\mathcal H}$.  As $1\leq rank\, Q \leq 2$, by using a unitary transformation 
we can assume that $Q$ is a diagonal matrix  of the form
\[
Q=
\left(
\begin{array}{ccc}
0 & &  \\
& \alpha & \\
&  & \beta
\end{array}
\right)
\hspace{.5cm} \alpha,\beta \geq 0, \, \, \alpha+\beta=1.
\]
If $Q$ is the closest point to $\frac{1}{3}I$, as it belongs to the segment determined by two points of 
$\mathbb{C}{\mathrm P}^2$, it follows that $Q$ is the midpoint of that segment,

\[
Q=
\left(
\begin{array}{ccc}
0 & &  \\
& \frac{1}{2} & \\
&  & \frac{1}{2}
\end{array}
\right).
\]
Finally, by direct computation we get
$|Q-\frac{1}{3}I|= \frac{\sqrt{3}}{3}$.
\end{proof}

\subsection{The center of mass}
\label{center}

\hspace{.2cm}

\vspace{.2cm}

Let $(\Sigma,ds^2)$ be a Riemannian surface of genus $3$ which is not conformally hyperelliptic. 
From \S\ref{quartic} we can assume that $\Sigma\subset \mathbb{C}{\mathrm P}^2$ is a quartic curve 
and $ds^2$ is conformal to the induced metric $ds^2 = e^{2u}\langle,\rangle$, $u\in C^\infty(\Sigma)$. 
The Riemannian measures are related in the same way $d\mu = e^{2u}d\Sigma$. 

\vspace{.2cm}

Let $P\in int\,{\mathcal H}$, $f_P:\mathbb{C}{\mathrm P}^2\longrightarrow\mathbb{C}{\mathrm P}^2$ 
the associated projectivity (\ref{projectivity}) and  $\Sigma_P=f_P(\Sigma)$ be the image quartic curve. 
The corresponding induced Riemannian surfaces $(\Sigma_P, \langle,\rangle_P)$ are all conformally 
equivalent and their Riemannian measures $d\Sigma_P$ give the same area.

\vspace{.2cm}

We consider the inclusion map of $\Sigma_P$ and its Gauss map
\[
A_P:\Sigma_P\longrightarrow {H\hspace{-.03cm}M}_1\hspace{-.03cm}(3),
\hspace{1cm}
B_P:\Sigma_P\longrightarrow {H\hspace{-.03cm}M}\hspace{-.03cm}(3).
\]
These objects allow us to define, for $a\in\mathbb{R}$, the {\it center of mass} map
\[
\Phi_a:int\,{\mathcal H}\longrightarrow {H\hspace{-.03cm}M}_1\hspace{-.03cm}(3),
\]
as the mean value of the map $\phi_a$ associated to the curve $\Sigma_P$,

\vspace{-.2cm} 

\[
\Phi_a(P) = \frac{1}{Area(ds^2)}\int_{\Sigma_P}\big( A_P + a\hspace{.02cm} B_P \big)\, d\mu.
\]

\vspace{.1cm} 

 Note that, if $P=\frac{1}{3}I$, from (\ref{phia}) we have

\[
\Phi_a(\frac{1}{3}I)= \frac{1}{Area(ds^2)}\int_{\Sigma}\phi_a  d\mu, \hspace{1cm}
\Phi_a(\frac{1}{3}I) - \frac{1}{3}I= \frac{1}{Area(ds^2)}\int_{\Sigma}\big( \phi_a - \frac{1}{3}I\big)\, d\mu.
\]

\vspace{.1cm}

The map $\Phi$ depends continuously on $a$ and $P$, and for $a=0$ it satisfies the following, 
see \cite{bourgui}:   

\vspace{.1cm} 

\begin{itemize}
  \item[$\cdot$]  The map $\Phi_0$ extends continuously to the compact convex body  $\mathcal H$, 
  $\Phi_0:\mathcal H\longrightarrow {H\hspace{-.03cm}M}_1\hspace{-.03cm}(3)$.

\vspace{.1cm}
 
 \item[$\cdot$] This extension maps the boundary to itself and the restriction 
  $\Phi_0:\partial {\mathcal H}\longrightarrow \partial {\mathcal H}$ has non zero topological degree.
\end{itemize}

 These facts play an essential role in the argument of \cite{bourgui} and they will in ours, too. 

\begin{lemma}  
\label{mass}
Let $ds^2$ be a conformal metric on the quartic curve $\Sigma\subset {\mathbb C}{\mathrm P}^2$. 
For  $0\leq a <\sqrt{3}/6$, the point $\frac{1}{3}I$ lies in the image of the center of mass map 
$\Phi_a:int\, {\mathcal H}\longrightarrow {H\hspace{-.03cm}M}_1\hspace{-.03cm}(3)$.
\end{lemma}

\begin{proof} 
As $\Phi_a$ is invariant under homothetic rescaling of the metrics $ds^2$, we will suppose that $Area(ds^2)=1$.
For $0 <\varepsilon < 1$  we consider the convex hypersurface

\[
{ S}_\varepsilon =
\{ (1-\varepsilon)Q + \frac{\varepsilon}{3}I \, / \, Q\in \partial {\mathcal H}\} \subset int \,{\mathcal H}
\]
and the bounded region $D_\varepsilon$ in $ {H\hspace{-.03cm}M}_1\hspace{-.03cm}(3)$ enclosed by 
${ S}_\varepsilon$,
 $\partial D_\varepsilon ={ S}_\varepsilon$.

\vspace{.2cm}

Given $a < \sqrt{3}/6$, we claim that for $\varepsilon$ small enough the image hypersurface
$\Phi_b({ S}_\varepsilon)$, $0\leq b \leq a$, does not contain the point $\frac{1}{3}I$. In fact,
reasoning by contradiction suppose that $\Phi_b(P)=\frac{1}{3}I$, for some $P\in { S}_\varepsilon$. Then 
\[
0=\Phi_b(P) -\frac{1}{3}I =
\Phi_0(P) -\frac{1}{3}I + b \int_\Sigma  B_P\, d\mu
\]
and therefore 
\begin{equation}
\label{xxx}
|\Phi_0(P)-\frac{1}{3}I | = b\left|\hspace{-.01cm}\int_\Sigma  B_P\, d\mu \right|  
\leq b\int_\Sigma |B_P|\, d\mu \leq 2a.
\end{equation} 
If we take $\varepsilon$ close to $0$, then
the point $P$ is close to $\partial {\mathcal H}$ and 
we have that $\Phi_0(P)$ is near $\partial{\mathcal H}$, \cite{bourgui}. In particular, 
the term at the left-hand-side of (\ref{xxx}) is larger than 
$2a$ and we have the desired contradiction. 

\vspace{.2cm}

Finally, we consider the continuous family of restricted maps
\begin{equation}
{\Phi_b}{\big|_{S_\varepsilon}}: S_\varepsilon\longrightarrow {H\hspace{-.03cm}M}_1
\hspace{-.03cm}(3)-\{\frac{1}{3}I\}, \hspace{1cm} 0\leq b \leq a.
\label{porfin}
\end{equation}

 As the degree of the ${\Phi_0}{\big|_{S_\varepsilon}}$ is non zero, \cite{bourgui}, 
by a well-known topological argument we get that $\frac{1}{3}I \in \Phi_a(D_\varepsilon)$ 
and we conclude the proof of the lemma. 
\end{proof}

\vspace{.05cm}

\begin{proof}[{\bf Proof of Theorem \ref{main}}]
 Let $(\Sigma,ds^2)$ be a compact Riemann surface of genus $g=3$ with a conformal metric.
If $\Sigma$ is hyperelliptic, then $\lambda_1(ds^2)A(ds^2)\leq 16\pi$, \cite{yy}. 
So henceforth we will assume $\Sigma$ is a smooth quartic algebraic curve in the complex projective 
plane $\Sigma \subset \mathbb{C}{\mathrm P}^2$.

For $a\in\mathbb{R}$ we consider the map (\ref{phia}),
\[
\phi_a:\Sigma\longrightarrow H\hspace{-.03cm}M_1\hspace{-.03cm}(3)\hspace{1cm} \phi_a = A+a\hspace{.03cm}B
\]
and we define the function
\begin{equation}
\label{F(a)}
F(a)=\frac{\int_\Sigma |\nabla \phi_a|^2 d\Sigma}
{
{\color{white}.}
\hspace{.3cm}|\phi_a-\frac{1}{3}I|^2\hspace{.3cm}
{\color{white}.}
}. 
\end{equation}

\vspace{.4cm}

{\bf Claim.} 
{\it a)} {\it For any quartic curve $\Sigma\subset \mathbb{C}{\mathrm P}^2$ we have}
\[
 \hspace{1cm} F(a)= \frac{7a^2-4a +1 }{3a^2-3a+1} 24\hspace{.03cm}\pi.
\]

\vspace{.2cm}

{\it b)}  {\it  The minimum of $F$ is attained for $a_1= (4-\sqrt{7})/{9}\approx 0.150$ and}
\[
 F(a_1)= (4-\sqrt{7})\hspace{.03cm}16\hspace{.03cm}\pi \approx 21.668\hspace{.03cm}\pi.
\]

\vspace{.2cm}

{\it Proof of the Claim.} a) In our  case $g=3$, $d=4$ and from Lemma \ref{lema1} we obtain

\[
\int_\Sigma |\nabla \phi_a|^2 d\Sigma= (7a^2-4a +1)\hspace{.03cm}32\hspace{.03cm} \pi. 
\]

\vspace{.1cm}

b) By direct computation we see that critical points of $F$ lie at $a=(4\pm\sqrt{7})/{9}$.
As  $\lim_{a\rightarrow \pm\infty} F(a)= 56\hspace{.03cm}\pi$, we conclude that 
the minimum of $F$ is attained for  $a_1=(4-\sqrt{7})/9$  
and is equal to $F(a_1)= (4-\sqrt{7})\hspace{.03cm}16\hspace{.03cm}\pi$. 
  
\vspace{.2cm}

{Now we conclude the proof of the theorem.} 

\vspace{.2cm}

As $a_1< \sqrt{3}/6$, from Lemma \ref{mass} there is a point $P\in int\, {\mathcal H}$ such that 
$\Phi_{a_1}(P)=\frac{1}{3}I$. Changing $\Sigma$ by its projective image, we can assume that  $P$ 
is proportional to the identity matrix and therefore 
\[
\int_\Sigma (\phi_{a_1}-\frac{1}{3}I)\, d\mu = 0.
\]
Using the linear coordinates of the map $\phi_{a_1}-\frac{1}{3}I$
as test functions for the first eigenvalue of $ds^2$ we obtain 
that $\lambda_1(ds^2)$ is smaller than or equal to the quotient of the energy of 
$\phi_{a_1}$ with respect to $ds^2$ over the integral of the length square of this vector valued map. 
As the total energy $\int_\Sigma |\nabla u|^2 d\mu$ of a function $u$ does not depend on the metric 
in the conformal class, by using the induced one $\langle,\rangle$ we have 
\[
\lambda_1(ds^2)  \leq 
\frac{\int_\Sigma |\nabla \phi_{a_1}|^2 d\Sigma}
{
{\color{white}.}
\hspace{.2cm}
\int_\Sigma |\phi_{a_1}-\frac{1}{3}I|^2 d\mu\hspace{.2cm}
{\color{white}.}
}
= \frac{F(a_1)}{\hspace{.1cm}Area(ds^2)\hspace{.1cm}}.
\]
Finally we obtain
\[
\lambda_1(ds^2) Area(ds^2) \leq F(a_1)
\]
and the theorem follows from the Claim above. 
\end{proof}

\subsection{Minimal surfaces of genus $3$ in manifolds with nonnegative Ricci curvature.} 

\hspace{.2cm}

\vspace{.2cm}

Let  $M^3$ an orientable Riemannian $3$-manifold with Ricci curvature bigger than or equal to zero, 
$Ric\geq 0$ and $\Sigma$ be a compact $2$-sided minimal surface of genus $g$ immersed in $M$. Let 
$\sigma$ the second fundamental form of the immersion and $d\Sigma$ be the induced measure. 
The Jacobi operator of the minimal immersion is, see \cite{ros} p. 72, 
\begin{equation}
\label{jacobi}
L = \Delta + Ric(N) +|\sigma|^2  =
\Delta  + Ric(e_1)+Ric(e_2) -2 K, 
\end{equation}
where $e_1,e_2$ is an orthonormal basis of the tangent plane of $\Sigma$ and $N$ is a unit normal vector. 

\vspace{.2cm}
We consider the eigenvalues of the Jacobi operator,
\[
\lambda_0<\lambda_1\leq \lambda_2 \leq \cdots.
\]

The eigenvalue $\lambda_0$ has multiplicity $1$ and the eigenfunction $\varphi_0$ will be assumed to be positive,
$L\varphi_0+\lambda_0\varphi_0 =0$.
The second variation formula of the area is given by the quadratic form 
\[
\label{q1}
Q(v,v)= -\int_\Sigma v Lv \hspace{.03cm} d\Sigma = 
\int_\Sigma \left\{|\nabla v|^2 -Ric(N) v^2 -|\sigma|^2 v^2\right\}  \hspace{.02cm} d\Sigma =
\]
\[
\int_\Sigma \left\{|\nabla v|^2 -\big(Ric(e_1)+Ric(e_2)\big)v^2 +2Kv^2 \right\}\hspace{.02cm} 
d\Sigma, \hspace{1cm} \forall v\in C^\infty(\Sigma).
\]

\vspace{.1cm}

Therefore, as $Ric\geq 0$, 
\[
 Q(v,v) \leq \int_\Sigma \big\{|\nabla v|^2 + 2Kv^2\big\} \hspace{.02cm}d\Sigma 
\]
and taking the function $v=1$ we obtain 
\begin{equation}
\label{lambda11}
\lambda_0 Area(\Sigma) \leq Q(1,1)\leq 2\int_\Sigma K \hspace{.02cm}d\Sigma = 8\pi (1-g). 
\end{equation}

If $g=3$, then from Lemma \ref{mass} we have an spherical map $\phi_{a_1}:\Sigma\longrightarrow HM_1(3)$
such that $\phi_{a_1}-\frac{1}{3}I$ is orthogonal to $\varphi_0$,
\[
\int_\Sigma (\phi_a -\frac{1}{3}I)\varphi_0\hspace{.03cm} d\Sigma = 0.
\]
So, $\lambda_1$ satisfies
\[
\lambda_1 \int_\Sigma \big|\phi_{a_1}-\frac{1}{3}I\big|^2 d\Sigma  
\leq Q\big(\phi_{a_1}-\frac{1}{3}I,\phi_{a_1}-\frac{1}{3}I\big) \leq 
\int_\Sigma |\nabla \phi_{a_1}|^2 d\Sigma  + 2\int_\Sigma K \big|\phi_{a_1}-\frac{1}{3}I\big|^2d\Sigma.
\]

\vspace{.2cm}
 
Thus, dividing by the constant $\big|\phi_{a_1}-\frac{1}{3}I\big|^2$, using Theorem \ref{main} and the Gauss-Bonnet 
theorem, we get

\begin{equation}
\label{lambda2}
\lambda_1 Area (\Sigma) \leq F(a_1) +2\int_\Sigma K \hspace{.02cm} d\Sigma \leq 16(4-\sqrt{7})\pi  -
16\pi = 16(3-\sqrt{7})\pi. 
\end{equation}

\vspace{.2cm}
Finally, from (\ref{lambda11}) and (\ref{lambda2}), we obtain the following bound for the first eigenvalues of the Jacobi operator of the minimal surface $\Sigma$.

\begin{corollary} Let $\Sigma$ be a compact two-sided minimal surface of genus $3$ immersed in an orientable Riemannian 
$3$-manifold $M$ with nonnegative Ricci curvature $Ric \geq 0$. If 
$\lambda_0$ and $\lambda_1$ are the first two eigenvalues of its Jacobi operator, then
\[
\lambda_0 \hspace{.02cm}Area(\Sigma) \leq -16 \hspace{.02cm}\pi,
\]
\[
\lambda_1 \hspace{.02cm}Area(\Sigma) \leq 16(3-\sqrt{7}) \hspace{.02cm}\pi \approx 5.668 \hspace{.02cm}\pi.
\]
\label{minimal}
\end{corollary}

It is shown in Ros \cite{ros} that if $g\geq 4$, then $\lambda_1$ is necessarily negative. 
For $g=3$, although the eigenvalue should be nonpositive for most of the surfaces, Ros \cite{ros2} 
constructed a minimal surface with $\lambda_1>0$ embedded in a  Riemannian $3$-dimensional projective 
space of positive curvature. On the other hand, Ketover, Marques and Neves \cite{ketovermarquesneves} prove that any compact orientable $3$-manifold with $Ric > 0$, admits a  minimal surface of genus $g\leq 2$ with $\lambda _1>0$ 
(the restriction on the genus follows from Hamilton \cite{hamilton}).   
So, the case $g=3$ is a borderline one and it would be of interest to improve the control of the invariant 
$\lambda_1 \hspace{.02cm}Area(\Sigma)$  given in Corollary \ref{minimal}.


{\footnotesize
\noindent
Department of Geometry and Topology and \vspace{-.12cm}\\
Institute of Mathematics (IEMath-GR), \vspace{-.12cm}\\
University of Granada\vspace{-.12cm}\\
 18071 Granada, Spain.\vspace{-.12cm}\\
{{\it Email address:} {\tt aros@ugr.es}} }

\end{document}